\documentclass{article}
\title{Contagions in Random Networks with Overlapping Communities}
\author{Emilie Coupechoux, {\textit{Universit\'{e} Nice Sophia Antipolis \footnote{This paper is part of the author's PhD thesis done at INRIA-ENS}}}\\
  Marc Lelarge, {\textit{INRIA - ENS}}\\
E-mail: Emilie.Coupechoux@unice.fr, Marc.Lelarge@ens.fr}

\usepackage{color} 
\usepackage{multirow} 
\usepackage{amssymb,bm,amsmath,dsfont,amsthm}
\usepackage{latexsym}
\usepackage{epsfig}
\usepackage{url}
\usepackage{mdwlist}
\usepackage{graphicx}

\allowdisplaybreaks

\def\Pb{\mathds{P}}
\def\Eb{\mathds{E}}
\def\Indb{\mathds{1}}

\newcommand{\cS}{{\mathcal S}}

\def\tD{\tilde{D}}

\def\bd{\bar{d}}
\def\bw{\bar{w}}

\def\boldp{\boldsymbol{p}}
\def\boldq{\boldsymbol{q}}

\def\pext{p_{\textrm{ext}}}
\def\pact{p_{\textrm{f}}}
\def\Gact{G^{(q)}_\textrm{act}}

\newcommand{\ie}{{\it i.e.~}}
\newcommand{\iid}{{\em i.i.d.~}}

\def\BP{\Gamma(\boldp,\boldq)}
\def\GBP{G_{\Gamma}(\boldp,\boldq)}
\def\BG{\mathcal{B}(\boldp,\boldq)}
\def\GBG{G_{\mathcal{B}}(\boldp,\boldq)}
\def\tW{\tilde{W}}
\def\tD{\tilde{D}}

\newtheorem{theorem}{Theorem}
\newtheorem{lemma}[theorem]{Lemma}
\newtheorem{proposition}[theorem]{Proposition}

\begin{document}

\maketitle





\begin{abstract}
%

We consider a threshold epidemic model on a clustered random graph with overlapping communities. In other words, our epidemic model is such that an individual becomes infected as soon as the proportion of her infected neighbors exceeds the threshold $q$ of the epidemic. In our random graph model, each individual can belong to several communities. The distributions for the community sizes and the number of communities an individual belongs to are arbitrary. 

We consider the case where the epidemic starts from a single individual, and we prove a phase transition (when the parameter $q$ of the model varies) for the appearance of a cascade, \ie when the epidemic can be propagated to an infinite part of the population. More precisely, we show that our epidemic is entirely described by a multi-type (and alternating) branching process, and then we apply Sevastyanov's theorem about the phase transition of multi-type Galton-Watson branching processes. In addition, we compute the entries of the matrix whose largest eigenvalue gives the phase transition.
\end{abstract}

\textbf{MSC classes:} 60C05, 05C80, 91D30


\section{Introduction}







The spread of diseases or e-mail viruses is well modeled by classical (SI, SIR or SIS) epidemics, whose study on complex networks has attracted a lot of attention in recent years (see Newman \cite{new-book03} for a review).
In such epidemics, each node can be independently influenced by each of her neighbor. For the diffusion of an innovation, individual's adoption behavior is highly correlated with the behavior of her neighbors \cite{vega07}, and threshold epidemic models are more appropriate to model such diffusions. In this paper, we consider the game-theoretic contagion model, proposed by Blume \cite{bl95} and Morris \cite{mor}, and described below.

Consider a graph $G$ in which the nodes are the individuals in the population and there is an edge $(i,j)$ if $i$ and $j$ can interact with each other. Each node has a
choice between two possible actions labeled $A$ and $B$. On each
edge $(i,j)$, there is an incentive for $i$ and $j$ to have their
actions match, which is modeled as the following coordination game
parametrized by a real number $q\in (0,1)$:
if $i$ and $j$ choose $A$ (resp. $B$), they each receive a payoff of
$q$ (resp. $(1-q)$); if they choose opposite actions, then they
receive a payoff of $0$.
Then the total payoff of a player is the sum of the payoffs with each
of her neighbors. If the degree of node $i$ is $d_i$ and $S_i^B$ is
 the number of her neighbors playing $B$, then the payoff to $i$ from
choosing $A$ is $q(d_i-S_i^B)$ while the payoff from choosing $B$ is
$(1-q)S^B_i$. Hence, in a best-response dynamic, $i$ should adopt $B$ if $S_i^B>qd_i$ and $A$ if
$S_i^B\leq qd_i$.
A number of qualitative insights can be derived from such
a model even at this level of simplicity \cite{klein} \cite{vega07}. Specifically,
consider a network where all nodes initially play $A$.
If a small number of nodes are forced to adopt strategy $B$ (the seed) and we apply
best-response updates to other nodes in the network, then these nodes will be
repeatedly applying the following rule: switch to $B$ if enough of
your neighbors have already adopted $B$. There can be a cascading
sequence of nodes switching to $B$ such that a network-wide
equilibrium is reached in the limit. 
In this paper, we consider the case where a node playing $B$ is forced to play $B$ forever (thus the number of players $B$ is non-decreasing) and where the seed consists of only one vertex. The graph $G$ is infinite, and we are interested in the cascade phenomenon, \ie when an infinite subset of the population will eventually adopt $B$. We will show a phase transition for this phenomenon, depending on the value of the parameter $q$ of the model.








When the graph $G$ is deterministic, such phase transitions were proved by Morris \cite{mor}. The graph $G$ that we consider here will have most of the properties observed in real-world networks. One of the most striking features shared by real-world networks is the scale-free property \cite{BA99}: their degree distribution follows a power law. Random graphs with an arbitrary degree distribution \cite{bc78:degseq} cover this property. The contagion model on such graphs was studied by heuristic means by Watts \cite{wat02}, and a generalization of it was studied rigorously by Lelarge \cite{lel:diff}. Another feature of real-world networks is that they all have a high clustering coefficient (Watts and Strogatz \cite{ws98}, see also Newman \cite{New03models} for several examples).
The clustering coefficient of a graph is by definition the probability that two given nodes are connected, knowing that they have a common neighbor. Since the asymptotic clustering coefficient of random graphs with an arbitrary degree distribution is zero (locally, they look like trees), this random graph model fails to cover the clustering property of real-world networks.
Recently, the contagion on clustered random graph models was studied by heuristic means by Gleeson \cite{Gleeson08:modular}, and rigorously in \cite{CL:netgcoop11} and \cite{arxiv12} (in which a generalization of the contagion model is considered). The random graph models considered have a tunable clustering coefficient and an arbitrary degree distribution, which in particular allows the study of the clustering impact on the contagion model. However, these random graph models do not cover the following property: in real-world networks, a node often belongs to several communities. A community is a set of nodes which are densely connected internally and only sparsely connected with other nodes of the network; in the clustered random graph models mentioned above, communities are represented by cliques, and a node can only belong to at most one clique. On the contrary, and as explained in \cite{NewSW01:GCbipartite} and \cite{GuillaumeLatapy06}, the structure of many real-world networks is close to the one-mode projection of a bipartite graph, in which each node belongs (possibly) to several cliques (communities).

A classical example of the one-mode projection of a bipartite graph is the collaboration graph of movie actors. Let  $\Gamma\subset V\times E$ be a bipartite graph, \ie a graph with two types of nodes: $V$-nodes and $E$-nodes, and in which there are no edges between nodes of the same type. Each $E$-node of $\Gamma$ represents a 'community' (a movie), and the $V$-nodes linked to a common $E$-node are members of the same community (actors of the same movie). The one-mode projection of $\Gamma$ on $V$-nodes (actors) is a unipartite graph: the nodes are the $V$-nodes of $\Gamma$, and there is an edge between two $V$-nodes if they belong to at least one common community (if these actors played together in at least one movie). One can construct a random graph by considering the one-mode projection of a random bipartite graph. Up to our knowledge, neither rigorous proofs nor heuristics have been done for the contagion on such a model. There are several random bipartite graph models, and, even the literature on the classical SIR epidemic on the one-mode projection of such models is incomplete. When the random bipartite graph has arbitrary degree distribution for each one of both types of nodes, heuristics have been derived by Newman \cite{new03} for the classical SIR epidemic on the one-mode projection of it. Rigorous results (for the classical SIR epidemic) were obtained by Britton \textit{et al.} \cite{britton-2007} and by Bollob\'{a}s \textit{et al.} \cite{BJR11:sparseRGwithC}, but the random graphs considered are such that their asymptotic degree distributions are respectively Poisson and mixed Poisson. 
Recently, Hackett \textit{et al.} \cite{HMGleeson:New} studied by heuristic means the contagion model on random graphs with overlapping communities, \ie with nodes that can belong to several cliques. They also derive results about the clustering effect on the contagion spread (for their random graph model). However, the communities in that random graph model are only of size three (that model does not come from the one-mode projection of a random bipartite graph).

Our random graph model is inspired from the one-mode projection of a random bipartite graph with arbitrary degree distributions. More precisely, we consider the one-mode projection of an alternating branching process that approximates locally this random bipartite graph (see \cite[Section 7.2]{DN08} for this approximation), and study rigorously the contagion on this random graph model.
Our goal is twofold: \textit{(i)} we study \textit{rigorously} the contagion on random graphs with overlapping communities; \textit{(ii)} our study provides heuristics for the contagion on the one-mode projection of random bipartite graphs with arbitrary degree distributions.

This paper is organized as follows. In Section~\ref{sec:modelBP}, we define our random graph model and recall its degree distribution and clustering coefficient. 
In Section~\ref{sec:contagionBP}, we state our theorem about the phase transition for the contagion spread on our random graph model with overlapping communities. In Section~\ref{sec:comput_M}, we compute the entries of the matrix involved in the phase transition. Finally, in Section~\ref{sec:proof_main_theorem}, we proved this phase transition, applying Sevastyanov's theorem about the phase transition of multi-type Galton-Watson branching processes.

\section{Random graph model and its basic properties} \label{sec:modelBP}

In this paper, we consider the one-mode projection of an alternating branching process. We define this branching process in Section~\ref{subs:def_BP}, and our random graph model in Section~\ref{subs:modelBP}. The degree distribution and clustering coefficient of our random graph model are the same as those computed in \cite[IV]{NewSW01:GCbipartite}, and are recalled in Section~\ref{subs:deg_clust}. 

Let $\boldp={(p_d)}_d$ and $\boldq={(q_w)}_w$ be two probability distributions with positive finite means $\lambda:=\sum_d d p_d\in(0,\infty)$ and $\mu:=\sum_w q_w \in(0,\infty)$. 

\subsection{Alternating branching process} \label{subs:def_BP}

Our random graph model is constructed from the following alternating branching process $\BP$, whose definition is given in Section~\ref{def_BP} and phase transition (for $\BP$ to be finite/infinite) recalled in Section~\ref{ph_tr_BP}.


\subsubsection{Definition} \label{def_BP}

The branching process $\BP$ is an alternating one: each node is either of type $V$ or $E$, and a generation of $V$-nodes gives birth to a generation of $E$-nodes, and conversely. 

Let $\tD$ and $\tW$ be independent random variables with the following distributions (for any $d\geq 1$, $w\geq 1$):
\begin{eqnarray*}
\Pb(\tD=d-1)=\frac{dp_{d}}{\lambda}, \quad
\Pb(\tW=w-1)=\frac{wq_{w}}{\mu}.
\end{eqnarray*}
The variable $\tD$ (resp. $\tW$) represents the offspring number of a $V$-node (resp. $E$-node), except for the root (whose offspring distribution is $\boldp$).  
The reasons why we choose this particular distribution are the following:
\begin{itemize}
\item This distribution gives a unimodular tree (the distribution is invariant by rerooting);
\item The branching process $\BP$ is a local approximation for the random bipartite graph $\mathcal{B}=\BG$ with arbitrary degree distributions $\boldp$ and $\boldq$ \cite[Section 7.2]{DN08}: informally, the root of the branching process $\Gamma$ represents a 'typical' vertex in $\mathcal{B}$. 
Hence our random graph model $G$ is a local approximation for the one-mode projection $\GBG$ of $\mathcal{B}$, and the root of $G$ represents a 'typical' vertex in $\GBG$. 
\end{itemize}

We can define formally the branching process $\Gamma=\BP$ as follows (in the following, only the definition of the $i$-th $V$ and $E$-generations are required). The root has type $V$ and offspring distribution $\boldp$: its number $D_0$ of children satisfies $\Pb(D_0=d)=p_d$. Let $(\tW_j^{(i)})_{i,j\geq 1}$ (resp. $(\tD_k^{(i)})_{i,k\geq 1}$) be random variables distributed as $\tW$ (resp. $\tD$), all variables being independent and independent from $D_0$. Then each child $j$ ($1\leq j\leq D_0$) of the root is an $E$-node that gives birth to $\tW^{(1)}_j$ $V$-nodes, so that the root has $\xi_V^{(1)}=\sum_{j=1}^{D_0} \tW^{(1)}_j$ grandchildren. Each such node $k$, $1\leq k\leq \xi_V^{(1)}$, is a $V$-node that gives birth to $\tD^{(1)}_k$ $E$-nodes. Set $\xi_E^{(1)}=\sum_{k=1}^{\xi_V^{(1)}} \tD^{(1)}_k$. Notations are summarized in Figure \ref{fig:NotationsBP}.

\begin{figure}
\centering
\scalebox{0.3}{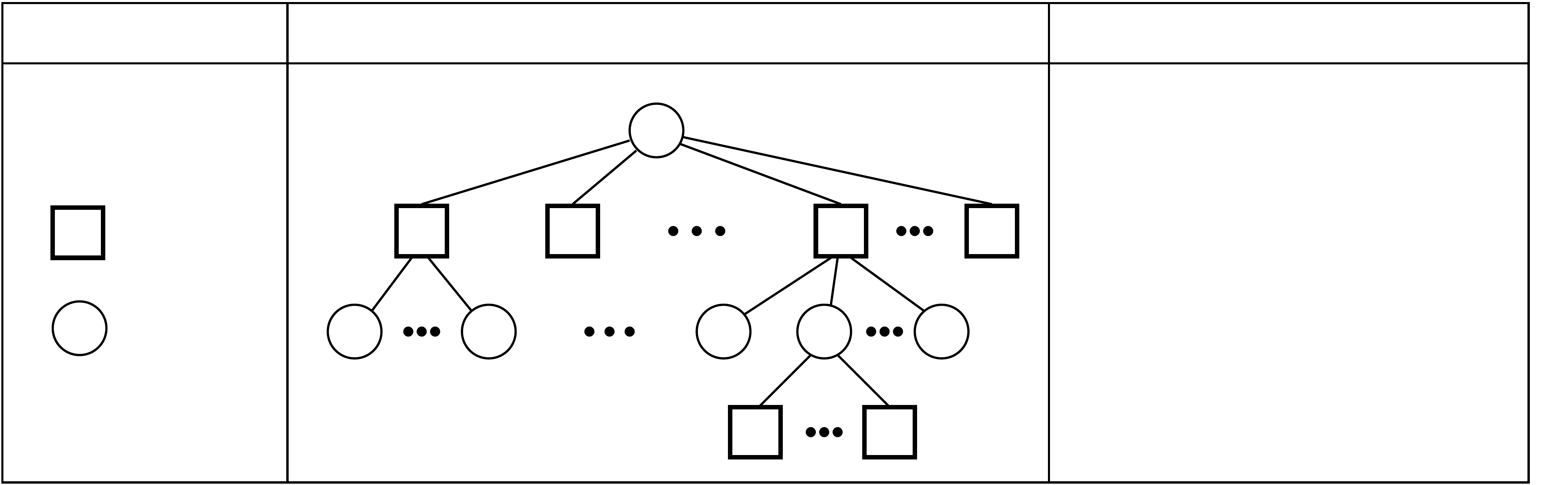}
\caption{Notations for the alternating branching process $\BP$.}
\label{fig:NotationsBP}
\end{figure}

The root corresponds to $V$-generation numbered $0$, its children to $E$-generation $0$, and so on (until now, we constructed $V$ and $E$-generations $0$ and $1$). Assume generations $0$ to $i-1$ are constructed ($i\geq 2$), with $\xi_E^{(i-1)}$ for the number of nodes in the $(i-1)$-th $E$-generation. Then each node $j$, $1\leq j\leq \xi_E^{(i-1)}$, gives birth to $\tW_j^{(i)}$ $V$-nodes, and we set $\xi_V^{(i)}=\sum_{j=1}^{\xi_E^{(i-1)}} \tW^{(i)}_j$ for the total number of nodes in the $i$-th $V$-generation. Each such node $k$, $1\leq k\leq \xi_V^{(i)}$, gives birth to $\tD_k^{(i)}$ $E$-nodes, and we set $\xi_E^{(i)}=\sum_{j=1}^{\xi_V^{(i)}} \tD^{(i)}_j$ for the total number of nodes in the $i$-th $E$-generation.

\subsubsection{Extinction vs. survival} \label{ph_tr_BP}


Let $D$ (resp. $W$) be a random variable with distribution $\boldp$ (resp. $\boldq$). 
We define the following generating functions, for ${x\in[0,1]}$:
\begin{eqnarray*}
 F(x) = \sum_{d\geq 0} p_d x^d \:, \;\;
 G(x) = \sum_{w\geq 1} \frac{wq_w}{\mu}x^{w-1} \:, \;\;
 H(x) = \sum_{d\geq 1} \frac{dp_d}{\lambda}x^{d-1} 
\end{eqnarray*}

The phase transition for $\BP$ to be finite/infinite is given by the next proposition, which is a direct consequence of Theorem 1 in \cite[Section I.A.5]{atne} (heuristics can also be found in \cite[IV.A]{NewSW01:GCbipartite}):
\begin{proposition} \label{prop:infiniteBP}
 Let $\pext$ be the probability that the branching process $\BP$ is finite.
\begin{itemize}
\item If $\Eb[W(W-1)]\Eb[D(D-1)]\leq \Eb[W]\Eb[D]$, then $\pext=1$;
\item If $\Eb[W(W-1)]\Eb[D(D-1)]>\Eb[W]\Eb[D]$, then $\pext<1$. More precisely, we have:
 $$\pext=(F\circ G)\left(\eta\right),$$ where $\eta:= \inf \left\{x\in [0,1]:(H \circ G)(x)=x \right\}$.
 \end{itemize}
\end{proposition}

\subsection{Random graph model} \label{subs:modelBP}


We construct the (rooted and unipartite) random graph $G=\GBP$ as follows. The root of $G$ is the root of $\Gamma=\Gamma(\boldp,\boldq)$, and
the parent and the children of each $E$-node $e$ in $\Gamma$ are connected into a clique before $e$ is removed, as illustrated in Figure \ref{fig:Brothers}. 
In other words, the random graph $\GBP$ is the one-mode projection of $\Gamma$. 


\begin{figure}
\centering
\scalebox{0.33}{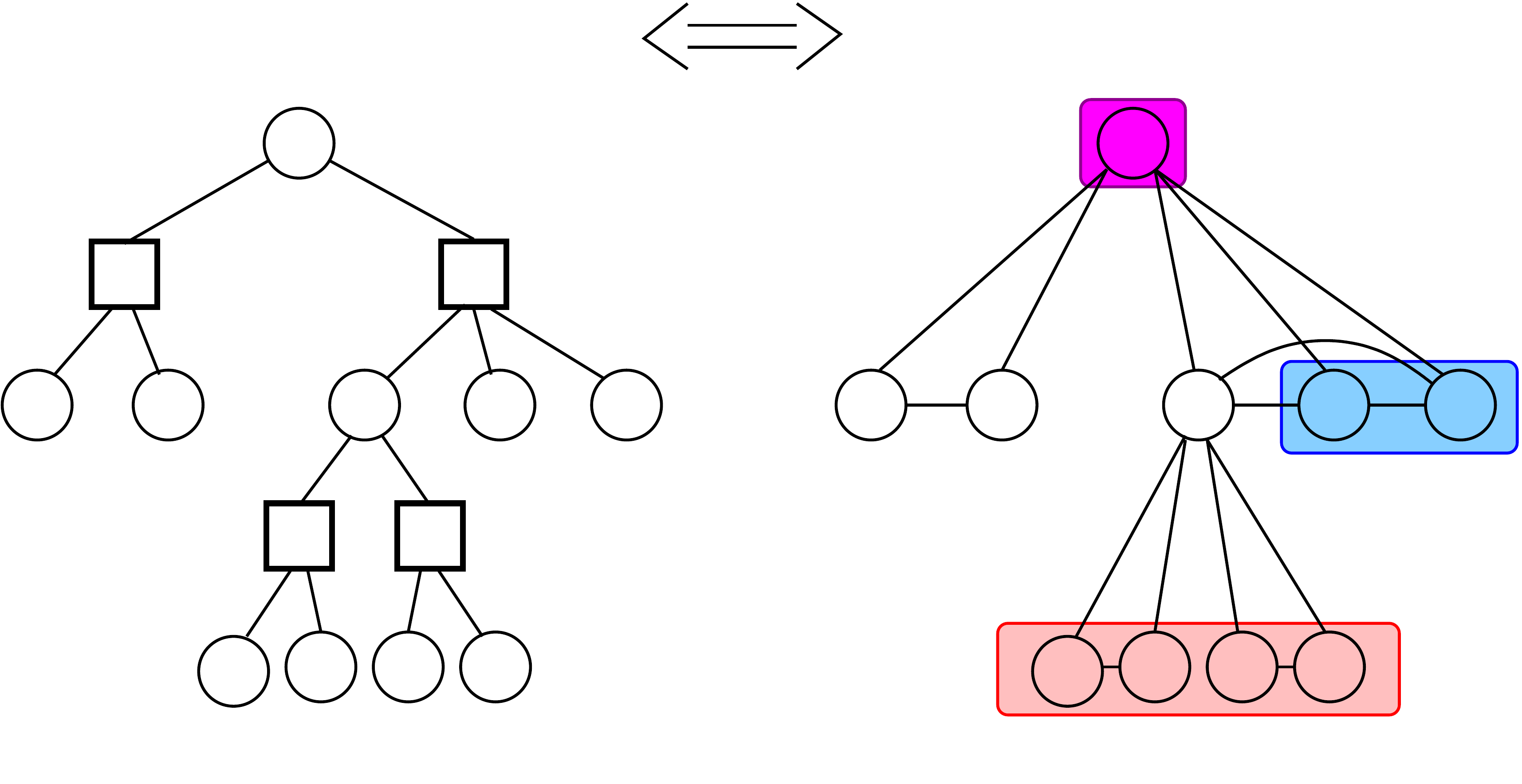}
\caption{Random graph model $\GBP$, constructed from the branching process $\BP$.}
\label{fig:Brothers}
\end{figure}

Even if the graph $G$ is not a tree, its particular construction allows using the same terminology as for a tree (see Figure  \ref{fig:Brothers}). More precisely, let $v$ be a node in $G$, and let $d$ its distance from the root. By construction, there is exactly one node $u$, among the neighbors of $v$, which is at distance $d-1$ from the root: it is called the \textit{parent} of $v$. In addition, the neighbors of $v$ that are at distance $d$ from the root are called the \textit{brothers} of $v$, and those at distance $d+1$ the \textit{children} of $v$. 
If we consider a given clique of $G$, the \textit{parent of the clique} is the node whose distance from the root is minimal.

Note that the random graph $G$ is infinite if and only if $\Gamma$ is, so that the phase transition for $G$ to be finite/infinite is given by Proposition \ref{prop:infiniteBP} above.


\subsection{Degree distribution and clustering coefficient} \label{subs:deg_clust}



The next proposition can be found in the paper of Newman \cite{new03}.

\begin{proposition}[\cite{new03}] \label{prop:deg_root_BP}
We consider the random graph $\GBP$. Then the degree distribution $D'_0$ of the root is given by:
 \begin{eqnarray*} \label{eq:deg_root_BP}
 \Pb(D'_0=k) &=& \sum_{d=1}^{\infty} p_d \sum_{w_1+\dots+w_{d}=k+d} \frac{\prod_{j=1}^{d} w_j q_{w_j}}{\mu^d}.
\end{eqnarray*}

Let $T_0$ be the number of triangles the root belongs to, and let $P_0$ be the number of connected triples the root belongs to. 
Then the local clustering coefficient $C$ of the root is given by:
\begin{eqnarray*} \label{eq:clust_BP}
C :=\frac{\Eb\left(T_0\right)}{\Eb\left(P_0 \right)} 
= \frac{\frac{1}{\Eb[W]} \Eb[W(W-1)(W-2)]}{\frac{\Eb[D(D-1)]}{\Eb[D]}  \left( \frac{\Eb[W(W-1)]}{\Eb[W]} \right)^2 + \frac{1}{\Eb[W]} \Eb[W(W-1)(W-2)]},
\end{eqnarray*}
where $D$ (resp. $W$) is a random variable with distribution $\boldp$ (resp. $\boldq$).
\end{proposition}

\section{Phase transition for the contagion: statement of our result} \label{sec:contagionBP}

In Section~\ref{subs:epidemic_model}, we define the contagion model described in the introduction, on the random graph $G=\GBP$. In Section~\ref{subs:ph_tr}, we state our main theorem (whose proof is given in Section~\ref{sec:proof_main_theorem}).

\subsection{Epidemic model} \label{subs:epidemic_model} \label{def_Gact}

Let $q\in(0,1)$. We consider the contagion model described in the introduction, on the random graph $G=\GBP$, with parameter $q$ and the seed consisting in the root only. For simplicity, players $B$ are called active vertices, and players $A$ inactive vertices. 
The progressive dynamics of the contagion on the graph $G$ operates as follows: the root starts out being active; all other vertices are inactive.
Time operates in discrete steps $t=1,2,3,\dots$. At a given
time $t$, any inactive vertex becomes
active if its proportion of active neighbors is strictly greater than $q$. Once active, a vertex stays active. Hence the set of active vertices increases with time, and we define $\Gact$ as the graph induced by the vertices in $G$ that are active in the limit as time tends to infinity.
We say that a cascade occurs if the graph $\Gact$ is infinite. We will show a phase transition for the cascade phenomenon, \ie for the graph $\Gact$ to be finite/infinite.

\subsection{Phase transition} \label{subs:ph_tr}

In the following, we make the additional assumption that the degrees of nodes in $\BP$ are bounded: there exist $\bd,\bw\geq 1$ such that $p_d=0$ for $d> \bd$ and $q_w=0$ for $w> \bw$. In addition, we assume that $p_0=q_0=q_1=0$.

We consider the random graph $G=\GBP$ defined in Section \ref{subs:modelBP}. We define the \textit{type} of a vertex in $G$ as its number of children in $G$. For all $x_0,x\in\{0,1,\dots,\bd\bw\}$, we set $m_{x_0,x}$ for the mean number of \textbf{active} children of type $x$ of an active vertex of type $x_0$ in $G$. In other words, we consider an active vertex $u$, different from the root, with $x_0$ children (if the probability that there exists such one is zero, then we set $m_{x_0,x}=0$). Once the contagion has spread among all its children, we count the number of such children that are active \textbf{and} have exactly $x$ children. The mean of this quantity (among all possible realizations of $G$ for the children and grandchildren of $u$) is called $m_{x_0,x}$. The matrix $M={(m_{x_0,x})}_{0\leq x_0,x\leq \bd \bw}$ will be computed in Section \ref{sec:comput_M}.

We have the following phase transition for the graph $\Gact$ to be finite/infinite, that will be proved in Section~\ref{sec:proof_main_theorem}.
\begin{theorem} \label{main_theorem}
Let $\pact$ be the probability that the random graph $\Gact$ is finite, and let $\rho$ be the largest eigenvalue of $M$.
Then we have:
\begin{enumerate}
 \item[(i)] if $q\geq 1/2$, then $\pact=1$;
 \item[(ii)] if $q<1/2$, then
\begin{itemize}
\item if $p_2=q_2=1$, then the random graph $\Gact$ is infinite with probability one;
\item otherwise $\pact=1$ if and only if $\rho\leq 1$.
\end{itemize}
\end{enumerate}

\end{theorem}

The idea developed in Sections \ref{sec:comput_M} and \ref{sec:proof_main_theorem} is the following.
We first describe the epidemic (\ie the graph $\Gact$) by a branching process. The construction of the random graph $G$ implies independence and recursion properties that left us with the study of the contagion spread inside one clique. Yet this propagation depends on the degree of the nodes (cf. Lemma~\ref{lem:def_L} in Section~\ref{subs:lemma_cont_BP}). In other words, the number of active vertices at a given generation depends on the offspring number of these vertices. Our idea is to use a multi-type (and alternating) branching process to encode this information (cf. Section~\ref{subs:alternatingBP_cont}). The definition of this branching process requires not only the computation of the number of active vertices inside a clique, but rather the computation of the joint distribution for this number and the types of the vertices (cf. Lemma~\ref{lem:comput_L} in Section~\ref{subs:lemma_cont_BP}).

Once the epidemic has been described by a multi-type branching process, we use a phase transition theorem for multi-type branching processes. As developed in Sections \ref{subs:Sevastyanov} and \ref{subs:main_pf}, classical results about multi-type branching processes do not apply directly in our case, since the matrix $M$ defined above is not positively regular. We thus use Sevastyanov's theorem. To complete the proof of Theorem~\ref{main_theorem}, we compute the final classes (defined in Section~\ref{subs:Sevastyanov}) of our multi-type branching process.

We divided our work into two parts: although Lemmas \ref{lem:def_L} and \ref{lem:comput_L} are used in both the proof of Theorem~\ref{main_theorem} and the computation of the matrix $M$, we chose to gather the computational part of our work in the next section.


%
%

\section{Computation of the matrix $M$} \label{sec:comput_M}

Before computing the entries of the matrix $M$ defined in Section~\ref{subs:ph_tr}, we state two lemmas about the number of active vertices inside a given clique. These lemmas will also be used in Section~\ref{subs:alternatingBP_cont} (to define the multi-type branching process that describes the contagion spread in $\GBP$).

\subsection{Lemmas about the number of active vertices inside a given clique} \label{subs:lemma_cont_BP}

%

 We consider a clique of size $w$ whose parent $u$ becomes active at time $t$. The other nodes inside the clique can be activated until time $t+w-1$. More precisely, the first nodes (different from $u$) that are possibly activated inside the clique are those with fewer children. If such activations occur at time $t+1$, then it can turn other nodes of the clique into active ones at time $t+2$, and so on.

The following lemma computes the number of active vertices inside a clique, when the contagion has spread inside the clique.
\begin{lemma} \label{lem:def_L}
We consider the contagion spread inside a clique of size $w$, whose parent $u$ is initially active. Let $L$ be the final number of active children (active vertices different from $u$) inside the clique. We denote by $\{1,\dots,w-1\}$ the set of children inside the clique, and set $X_i$ for the number of children of vertex $i$, $1\leq i\leq w-1$ (so that $X_{(i)} + w-1$ is the degree of vertex $i$). Then $L$ satisfies the following equation: 
\begin{eqnarray}\label{eq:L}
 L=\min \Big\{i\in\{1,\dots,w-1\} \Big| \lfloor q (X_{(i)} + w-1) \rfloor+1 > i \Big. \Big\} - 1,
\end{eqnarray}
where $X_{(1)}=\min_i X_i\leq X_{(2)}\leq \dots \leq X_{(w-1)}=\max_i X_i$ is the order statistics of $(X_i)_{1\leq i\leq w-1}$, and $L=w-1$ if this set is empty (\ie if $\lfloor q (X_{(i)} + w-1) \rfloor+1\leq i$ for all $1\leq i\leq w-1$).
\end{lemma}

\begin{proof}
By definition, a node $i$, $1\leq i\leq w-1$, becomes active if and only if the proportion of its active neighbors is strictly greater than $q$, \ie if and only if its number of active neighbors is at least 
\begin{eqnarray} \label{eq:active_least}
 A_i:=\lfloor q \cdot  (X_i + w-1) \rfloor +1.
\end{eqnarray}

We use the order statistics of $(A_i)_{1\leq i\leq w-1}$ (or equivalently the one of $(X_i)_{1\leq i\leq w-1}$): nodes with fewer children need fewer active neighbors to become active, and the first node(s) to become possibly active is (are) the one(s) with $X_{(1)}$ children. More precisely, if $A_{(1)}>1$, no node different from $u$ inside the clique can be activated, and $L=0$. If $A_{(1)}\leq 1$, then at least one node (different from $u$) is activated. Then a second one is also activated if and only if $A_{(2)}\leq 2$, and the lemma follows by a simple induction.

\end{proof}

The next lemma provides the joint distribution of $L$ and the order statistics ${\left(X_{(i)}\right)}_{1\leq i\leq L}$, given the size $w$ of the clique.
The random variables $(X_i)_{1\leq i\leq w-1}$ are \iid and distributed as the following random variable $X$, which represents the number of children of a vertex different from the root in $\GBP$ (this equation is similar to the one of Proposition~\ref{prop:deg_root_BP}):
\begin{eqnarray} \label{eq:gdchildren}
\Pb(X=k)= \sum_{d=1}^{\bd} \frac{d p_{d}}{\lambda}  \sum_{(w_1-1)+\dots+(w_{d-1}-1)=k} \frac{\prod_{j=1}^{d-1} w_j q_{w_j}}{\mu^{d-1}}.
\end{eqnarray}
For any sequence $ x_1\leq x_2\leq \dots\leq x_\ell $ and any $i\in\{1,\dots,\ell-1\}$, we set (omitting the dependency on $(x_1,\dots,x_{\ell})$):
\begin{eqnarray} \label{eq:sk}
 s_i:= \begin{cases}  \max\{j\geq 0 | x_i=x_{i+j}\}+1&\mbox{if } x_{i-1}< x_i \mbox{ or } i=1, \\
1 & \mbox{if } x_{i-1}= x_i, \end{cases}
\end{eqnarray}
When $x_{i-1}< x_i$ (or $i=1$), $s_i$ is the number of $i'\geq i$ such that $x_{i'}=x_i$. For instance, if $\ell=6$ and $x_1<x_2=x_3=x_4<x_5=x_6$, then $s_1=1,s_2=3,s_3=s_4=1$ and $s_5=2$, so that $\prod_{i=1}^{5} s_i !=3 ! \, 2!$. 

\begin{lemma} \label{lem:comput_L}
We consider a clique of size $w$ whose parent $u$ is initially active, and use the same notations as for Lemma \ref{lem:def_L}. Then we have, for any $1\leq \ell\leq w-1$ and $ 0\leq x_1\leq x_2\leq \dots\leq x_\ell$:
 \begin{eqnarray*}\label{eq:E-offspring}
 p^{(E)}(\ell,x_1,\dots,x_\ell|w) &:= & \Pb(L=\ell, X_{(1)}=x_1,\dots,X_{(\ell)}=x_\ell) \nonumber \\
 &= &\left(\prod_{i=1}^{\ell}\Indb_{\{\lfloor q (x_i + w-1) \rfloor+1\leq i\}}\right) \frac{(w-1)! }{(w-1-\ell)! \:\prod_{i=1}^{\ell-1}s_i !}  \nonumber \\
&& \cdot \left(\prod_{i=1}^\ell \Pb(X=x_i) \right) \cdot \Big(\Pb\left(\lfloor q (X + w-1) \rfloor> \ell\right)\Big)^{w-1-\ell} , \\ 
p^{(E)}(0|w)&:=&\Pb(L=0) \\
& =& \Big( \Pb(\lfloor q (X + w-1) \rfloor>0) \Big)^{w-1}.
\end{eqnarray*}
\end{lemma}

\begin{proof}
The joint distribution of the order statistics for a sequence of \iid discrete random variables ${(Y_i)}_{1\leq i\leq n}$, distributed as $Y$, is given by the following equation, for any $y_1\leq y_2\leq \dots\leq y_{n}$ (see \cite[Equation (2.3)]{OrderStat}):
\begin{eqnarray} \label{eq:dist_OS}
  \Pb(Y_{(1)}=y_1,\dots,Y_{(n)}=y_{n} ) &=&  \frac{n!}{\prod_{i=1}^{n-1} s_i !} \prod_{i=1}^{n} \Pb(Y=y_i) \\
&=: & p_Y(y_1,\dots,y_{n}),    \nonumber
\end{eqnarray}
where ${(s_i)}_i$ (defined in \eqref{eq:sk}) correspond to the sequence $(y_1, \dots, y_{n})$.

If $\lfloor q (x_i + w-1) \rfloor+1> i$ for some $1\leq i\leq \ell$, then $L<\ell$ (due to \eqref{eq:L}), so that $p(\ell,x_1,\dots,x_\ell|w)=0$. 
We assume that $\lfloor q (x_i + w-1) \rfloor+1\leq i$ for all $1\leq i\leq \ell$. Then we have, using equation \eqref{eq:L} for the random variable $X$ defined in \eqref{eq:gdchildren}: 
\begin{eqnarray}
 p^{(E)}(\ell,x_1,\dots,x_\ell|w)&=&\sum_{x_\ell < x_{\ell+1}\leq \dots\leq x_{w-1}}  p_X(x_1,\dots,x_{w-1}) \; \Indb_{\left\{\lfloor q(x_{\ell+1}+w-1)\rfloor>\ell\right\}} \nonumber\\
&=&\sum_{x_{\ell+1}\leq \dots\leq x_{w-1}}  p_X(x_1,\dots,x_{w-1}) \; \Indb_{\left\{\lfloor q(x_{\ell+1}+w-1)\rfloor>\ell\right\}} \label{eq:pE}
\end{eqnarray}
The last equality comes from the fact that $\lfloor q (x_\ell + w-1) \rfloor+1\leq \ell$.
We set:
\begin{eqnarray*}
g_X(\ell,w) &:= &\sum_{x_\ell < x_{\ell+1}\leq \dots\leq x_{w-1}} \frac{(w-1-\ell)!}{\prod_{i=\ell+1}^{w-2}s_i !} \prod_{i=\ell+1}^{w-1} \Pb(X=x_i)  \; \Indb_{\left\{\lfloor q(x_{\ell+1}+w-1)\rfloor>\ell\right\}} .
\end{eqnarray*}
Then we have, using \eqref{eq:pE} and replacing $p_X(x_1,\dots,x_{w-1})$ by its expression in \eqref{eq:dist_OS}:
\begin{eqnarray*}
 p^{(E)}(\ell,x_1,\dots,x_\ell|w)&=&\left( \frac{(w-1)!}{\prod_{i=1}^{\ell-1} s_i !} \prod_{i=1}^{\ell} \Pb(X=x_i) \right) \cdot \frac{g_X(\ell,w)}{(w-1-\ell)!}
\end{eqnarray*}
We now compute $g_X(\ell,w)$.
Let $Z_{1},\dots, Z_{w-1-\ell}$ be \iid random variables distributed as $X$, and set $Z_{(1)}:=\min_{1\leq k\leq w-1-\ell} Z_k$. Then we have, using  \eqref{eq:dist_OS}:
\begin{eqnarray*}
g_X(\ell,w)&=& \sum_{ x_{\ell+1}\leq \dots\leq x_{w-1}} p_X(x_{\ell+1},\dots,x_{w-1}) \Indb_{\left\{\lfloor q(x_{\ell+1}+w-1)\rfloor>\ell\right\}}\\
&=& \Pb\left(\lfloor q\cdot(Z_{(1)}+w-1)\rfloor>\ell\right) \\
&=& \Big(\Pb\left(\lfloor q (X + w-1) \rfloor> \ell\right)\Big)^{w-1-\ell}.  
\end{eqnarray*}

This proves the first equation, and the second one follows from
\begin{eqnarray*}
 \Pb(L=0) &=& \Pb(\lfloor q (X_{i} + w-1) \rfloor+1>1,  \textrm{ for all } i \leq w-1) \\
&=& \Big( \Pb(\lfloor q (X + w-1) \rfloor>0) \Big)^{w-1},
\end{eqnarray*}
which ends the proof.
\end{proof}

\subsection{The number of active vertices of a given type}

The next proposition gives the computation for the matrix $M$ 
(whose largest eigenvalue is used in Theorem~\ref{main_theorem}).
\begin{proposition}
We define the following quantities, for $0\leq x\leq \bd\bw$, $1\leq w\leq\bw$, $1\leq k\leq w-1$, $k+\lfloor q(x+w-1)\rfloor\leq \ell\leq w-1$ and $\lfloor q(x+w-1)\rfloor\leq i\leq \ell-k$ (with $p^{(E)}(\ell,x_1,\dots,x_\ell|w)$ defined in Lemma \ref{lem:comput_L}):
\begin{eqnarray*}
 \cS_{(x,k,\ell,i)}&:=& \left\{(x_1,\dots,x_{\ell})\left|x_1\leq\dots x_{i}<x_{i+1}=\dots=x_{i+k}=x<x_{i+k+1}\leq x_{\ell} \right.\right\} \\
P_{x|w}(k,\ell,i)  &:=& \sum_{(x_1,\dots,x_{\ell})\in \cS_{(x,k,\ell,i)}} p^{(E)}(\ell,x_1,\dots,x_\ell|w) \label{eq:P_x|w} \\
m_{x|w} &:=&\sum_{1\leq k\leq w-1} k \cdot \sum_{k+\lfloor q(x+w-1)\rfloor\leq \ell\leq w-1} \left( \sum_{\lfloor q(x+w-1)\rfloor\leq i\leq \ell-k} P_{x|w}(k,\ell,i) \right) \label{eq:m_x|w}
\end{eqnarray*}

Let $v$ be an active vertex in $G=\GBP$ with type $x_0$, $0\leq x_0\leq \bd\bw$. Then the mean number $m_{x_0,x}$ of its \textit{active} children having type $x$, $0\leq x\leq \bd\bw$, is $m_{0,x}=0$ if $x_0=0$ and otherwise is given by:
\begin{eqnarray}  \label{eq:m_x_0,x}
 m_{x_0,x}=\sum_{d=1}^{\bd} \frac{dp_{d}}{\lambda}  \sum_{(w_1-1)+\dots+(w_{d-1}-1)=x_0} \frac{\prod_{j=1}^{d-1} w_j q_{w_j}}{\mu^{d-1}} \sum_{i=1}^{d-1} m_{x|w_i}.
\end{eqnarray}
\end{proposition}

\begin{proof}

Using the definition of $p^{(E)}(\ell,x_1,\dots,x_\ell|w)$ in Lemma \ref{lem:comput_L}, one can easily see that $P_{x|w}(k,\ell,i) $ is the probability that a clique of size $w$, whose parent is initially active, has exactly $\ell$ active children, $i$ of which having type strictly less than $x$ and exactly $k$ of which having type $x$.

Hence the mean number, for a clique of size $w$, of its active children having type $x$ is given by:
\begin{eqnarray*}
N_{x|w}:=\sum_{1\leq k\leq w-1} k \cdot \sum_{0\leq \ell\leq w-1} \left( \sum_{0\leq i\leq \ell-k} P_{x|w}(k,\ell,i) \right).
\end{eqnarray*}
In addition, for the children of type $x$ to be activated inside the clique, there must be at least $\lfloor q(x+w-1)\rfloor+1$ active children of type strictly less than $x$ (due to \eqref{eq:active_least}). In other words, for $P_{x|w}(k,\ell,i) $ to be positive, one should have $i+1\geq\lfloor q(x+w-1)\rfloor+1$. Moreover, we have: $\ell \geq k+i\geq k+\lfloor q(x+w-1)\rfloor$, so that $N_{x|w}=m_{x|w}$. 

The end of the proof follows easily, since we can study independently two different cliques having the same parent: with the notations of Figure  \ref{fig:Brothers}, if the parent of $v$ is active, the fact that $v$ becomes active or not is independent of the activation/non-activation of $v'$, when $v'$ is not a brother of $v$.

\end{proof}

\section{Proof of Theorem \ref{main_theorem}} \label{sec:proof_main_theorem}

This section is organized as follows. First we define a multi-type (and alternating) branching process that completely describes the contagion spread in $G=\GBP$ (using Lemma~\ref{lem:comput_L} of Section~\ref{subs:lemma_cont_BP}).
Second we recall Sevastyanov's theorem about the phase transition for multi-type Galton-Watson branching processes \cite{Sevastyanov}. We then use this theorem to prove Theorem~\ref{main_theorem} (stated in Section~\ref{subs:ph_tr}).

\subsection{Description of the contagion model by a multi-type (and alternating) branching process} \label{subs:alternatingBP_cont}


We now define an alternating branching process $\Gamma'$ in which the number of $V$-nodes in the $i$-th $V$-generation is distributed as the number of vertices in the $i$-th generation of $\Gact$ (for any $i\geq 0$).
As explained in Lemma~\ref{lem:def_L}, the number $L$ of vertices at the end of the contagion spread inside a clique depends on:
\begin{itemize}
 \item the clique size $w$,
\item the number of children in $\GBP$ of each child $i$, $1\leq i\leq w-1$, inside the clique. 
\end{itemize}
These dependencies lead to considering a multi-type branching process for $\Gamma'$ (for an introduction on such processes, see for instance Harris \cite[Chapter 2]{Harris} or Mode \cite[Chapters 1 and 2]{Mode:MultitypeBP}). 


We need the following notation. For all $1\leq d\leq \bd$ and $ 0\leq w_1, w_2, \dots, w_{d-1}\leq \bw $, we set:
\begin{eqnarray} \label{eq:V-offspring}
 p^{(V)}(d-1,w_1,\dots,w_{d-1}|x) &:=& \Indb_{\left\{\sum_{i=1}^{d-1} (w_i-1)=x\right\}} \frac{dp_{d}}{\lambda} \left(\prod_{i=1}^{d-1} \frac{w_iq_{w_i}}{\mu} \right) \nonumber \\
&&\qquad \cdot \frac{x !}{(H\circ G)^{(x)}(0)} ,
\end{eqnarray}
where the generating functions $G$ and $H$ are defined in Section~\ref{ph_tr_BP}, and $(H\circ G)^{(x)}(0)$ is the value at point $0$ of the $x$-th derivative of $H\circ G$. 

\begin{proposition} \label{prop:GammaPrime}
We define an alternating and multi-type branching process $\Gamma'$ as follows. 
The root is a $V$-node that gives birth to $d$ ($\leq \bd$) $E$-nodes with probability $p_d$. Each of these $E$-nodes has type $w$ with probability ${wq_{w}/\mu}$. 

An $E$-node of type $w$ gives birth to $\ell\in\{0,\dots,w-1\}$ $V$-nodes of types $0\leq x_{1}\leq x_{2}\leq \dots \leq x_{\ell}\leq \bd  \bw$ with probability $p^{(E)}(\ell,x_1,\dots,x_\ell|w)$ defined in Lemma \ref{lem:comput_L}.

Finally, a $V$-node of type $x\geq 0$ gives birth to $d-1$  $E$-nodes of types $w_1, \dots, w_{d-1}$ with probability $p^{(V)}(d-1,w_1,\dots,w_{d-1}|x)$ defined above.

As before, we set $\Gamma=\BP$ for the alternating branching process defined in Section~\ref{subs:def_BP}, $G=\GBP$ for the corresponding random graph model, and $\Gact$ for the random graph of active vertices in $G$ (as defined in Section~\ref{def_Gact}).
Then there is a coupling between $\Gamma'$ and $\Gact$ such that: 
\begin{itemize}
\item the number of $V$-nodes in the $i$-th $V$-generation of $\Gamma'$ is distributed as the number of vertices in the $i$-th generation of $\Gact$ (for any $i\geq 0$);
\item the type $w$ attached to each $E$-node in $\Gamma'$ corresponds to its 'size' in $\Gamma$ (\ie its offspring number in $\Gamma$ is $w-1$);
\item the type $x$ attached to each $V$-node in $\Gamma'$ corresponds to its number of children in $G$ (\ie its number of {\em grandchildren} in $\Gamma$ is $x$).
\end{itemize}
\end{proposition}

\begin{proof}
Using Lemma~\ref{lem:comput_L} and the fact that the contagion spreads independently in two different cliques, we are left to prove that, when $\sum_{i=1}^{d-1} (w_i-1)=x$, the probability
\begin{eqnarray*}
P:= \Pb\left(\tD=d-1,\tW_1=w_1-1,\dots,\tW_{\tD}=w_{d-1}-1\left|\sum_{i=1}^{\tD-1} (\tW_i-1)=x\right.\right)
\end{eqnarray*}
is given by \eqref{eq:V-offspring}, where $\tD$ is the offspring number of a $V$-node in $\GBP$ (with generating function $H$), and ${(\tW_i)}_i$ the offspring numbers of $E$-nodes in $\GBP$ (with generating function $G$).

%
%

We have:
\begin{eqnarray*} \label{eq:N-and-D}
 P &=& \frac {\Pb(\tD=d-1,\tW_1=w_1-1,\dots,\tW_{\tD}=w_{d-1})}{\Pb\left(\sum_{i=1}^{\tD} \tW_i=x\right)}.
\end{eqnarray*}

The numerator is easy to compute:
\begin{eqnarray*} \label{eq:Numerator}
 \Pb(\tD=d-1,\tW_1=w_1-1,\dots,\tW_{\tD}=w_{d-1}) &=& \frac{dp_{d}}{\lambda} \prod_{i=1}^{d-1} \frac{w_i q_{w_i}}{\mu} 
\end{eqnarray*}

We use generating functions to compute the denominator. We have, for all $y\in[0,1]$:
\begin{eqnarray*}
 \Eb\left[y^{\sum_{j=1}^{\tD} \tW_j}\right] &=& \sum_{d=1}^{\bd} \Pb(\tD=d-1) \Eb\left[y^{\sum_{j=1}^{d-1} \tW_j}\right]\\
&=& \sum_d \frac{dp_{d}}{\lambda} \left(\Eb\left[y^{\tW}\right]\right)^{d-1} \\
&=& (H\circ G) (y) \\
&=& \sum_x \frac{1}{x !} (H\circ G)^{(x)}(0) y^x,
\end{eqnarray*}
which leads to:
\begin{eqnarray*}
 \Pb\left(\sum_{j=1}^{\tD} \tW_j=x\right) &=& \frac{1}{x !} (H\circ G)^{(x)}(0),
\end{eqnarray*}
and ends the proof.

\end{proof}


By definition, there is a cascade in $\GBP$ if $\Gact$ is infinite, which occurs if and only if the branching process $\Gamma'$ is infinite.

%
%
%

\subsection{Phase transition for multi-type Galton-Watson branching processes} \label{subs:Sevastyanov}

We refer to Harris \cite[Chapter 2]{Harris} for further information on multi-type Galton-Watson branching processes.
Let $Z$ be a multi-type Galton-Watson branching process, with $k$ types, starting from only one individual of a given type $i_0$ (results are the same when the process starts from a finite number of individuals, but we are interested only in this case here). Let $M={(m_{ij})}_{1\leq i,j\leq k}$ be the first-moment matrix of $Z$, \ie $m_{ij}$ is the mean number of children of type $j$ that is created by a single individual of type $i$. 

We next recall Sevastyanov's theorem about the phase transition for multi-type Galton-Watson branching processes, that we will use in the next section.
We define final classes as in \cite[Chapter 2]{Harris}. Let $m_{ij}^{(n)}$ be, for $n\geq 1$, the element in the $i$-th row and $j$-th column of $M^n$. The types $i$ and $j$ \textit{communicate} if $m_{ij}^{(n)},m_{ji}^{(n')}>0$ for some $n,n'\geq 1$. A type that communicates neither with itself nor with any other type is called \textit{singular}; a \textit{class} is a set of types, each pair of which communicate, that is not contained in any other set having this property. The types fall uniquely into singular types and mutually exclusive classes. A \textit{final class} $C$ is a class in which any individual of type $i\in C$ has probability one to give birth to exactly one individual with type in $C$ (other individuals whose type is not in $C$ may also be produced).
Let $\rho$ be the largest eigenvalue of $M$. Sevastyanov's theorem can be stated as follows:
\begin{theorem}[\cite{Sevastyanov}] \label{th:Sevastyanov}
 The probability of extinction of the branching process $Z$ is one if and only if
(a) $\rho\leq 1$ and
(b) there are no final classes.
\end{theorem}
The proof of this theorem originally appeared in a paper of Sevastyanov \cite{Sevastyanov}. The statement used here corresponds to \cite[Chapter 2, Theorem 10.1]{Harris}. 

{\bf Remark. }
There is a classical version of Theorem \ref{th:Sevastyanov} (with a simpler proof) for positively regular processes.
The process $Z$ is said to be {\em positively regular} if there exists $n\geq 1$ such that $M^n$ is positive, \ie all entries of $M^n$ are (strictly) positive. The particular case of Theorem \ref{th:Sevastyanov} for positively regular processes can be found in \cite[Chapter 2, Theorem 7.1]{Harris}, or in any book dealing with multi-type branching processes (see also Mode \cite[Chapter 1]{Mode:MultitypeBP} or Athreya and Ney \cite[Chapter 5]{atne} for instance). Some examples of non-positively regular processes are also considered in \cite[Chapter 2]{Mode:MultitypeBP}. The hypothesis that the branching process has no final class is reduced to the hypothesis of non-singularity in the positively regular case.
We will see in the proof of Theorem~\ref{main_theorem} that the process we consider is not positively regular as soon as $p_1>0$ for instance, so that we need the stronger result of Sevastyanov.

\subsection{Proof of Theorem \ref{main_theorem}} \label{subs:main_pf}

We construct a non-alternating branching process $\Gamma'_V$ from $\Gamma'$ by erasing $E$-generations, so that a $V$-generation directly gives birth to the next $V$-generation, and we apply Theorem~\ref{th:Sevastyanov} to the branching process $\Gamma'_V$.

Note that we have the following fact, as mentioned in the remark following Theorem~\ref{th:Sevastyanov}: if $p_1>0$, then the process $\Gamma'_V$ is not positively regular, since $m_{0x}=0$ for all $x$. In addition, we cannot assume without loss of generality that $p_1=0$. Indeed the $V$-nodes of degree $1$ in the original branching process $\Gamma=\BP$ play a role in the contagion inside a clique, since the activation of the nodes with high degree depends on the activation of the nodes with low degree (cf. Lemma~\ref{lem:def_L} in Section~\ref{subs:lemma_cont_BP}).

Case \textit{(i)} is obvious. We assume that $q<1/2$.
We will show that there is no final classes in $\Gamma'_V$ if and only if either $p_2<1$ or $q_2<1$. Then applying Theorem~\ref{th:Sevastyanov} ends the proof (the case $p_2=q_2=1$ is direct). 
We first show that, if there exists a final class $C$ in $\Gamma'_V$, then necessarily $C=\{1\}$. We start with a definition and a lemma.

%

\paragraph{\textbf{Definition.}}

Let $0<x\leq \bd \bw$. A \textit{configuration} starting from $x$ is an element $\sigma_x$ of the form 
\begin{eqnarray*}
 \sigma_x={\Big((w_i,x_{i1},\dots,x_{i\ell_i})\Big)}_{1\leq i\leq d-1},
\end{eqnarray*}
where $1\leq d\leq \bd$, $2\leq w_i\leq \bw$, $0\leq \ell_i\leq w_i-1$ and $0\leq x_{i1}\leq \dots\leq x_{i\ell_i}\leq \bd \bw$ for all $1\leq i\leq d-1$. We say that a configuration $\sigma_x$ has \textit{positive probability} to occur if
\begin{itemize}
\item $p_d>0$,
 \item $\sum_{i=1}^{d-1} (w_i-1) = x$,
\item $q_{w_i}>0$ and $p^{(E)}(\ell_i,x_{i1},\dots,x_{i\ell_i}|w_i)>0$ for all $1\leq i\leq d-1$.
\end{itemize}
In other words, a configuration $\sigma_x$ of positive probability is a possible realization for the next two generations starting from a $V$-node $v$ of type $x$ in $\Gamma'$ (it contains all the information about the children and grandchildren of $v$ in the sense that the $w_i$'s represent the types of the children of $v$ in $\Gamma'$, and $x_{ij}$ the types of its grandchildren). The next lemma provides a simple way to construct new configurations $\sigma'_x$ of positive probability, if one knows a given configuration $\sigma_x$ of positive probability (the proof is obvious, using the expression of $p^{(E)}(\ell_i,x_{i1},\dots,x_{i\ell_i}|w_i)$ given in Lemma \ref{lem:comput_L}):

\begin{lemma} \label{lem:positive_config}
Let $\sigma_x={\Big((w_i,x_{i1},\dots,x_{i\ell_i})\Big)}_{1\leq i\leq d-1}$ be a configuration that occurs with positive probability. Let ${(y_{ij})}_{i,j}$ be such that
$ y_{ij}\leq x_{ij}$ and $\Pb(X=y_{ij})>0$ for all $1\leq i\leq d-1$, $1\leq j\leq\ell_i$,
where the distribution of $X$ is given by \eqref{eq:gdchildren}.
Then the new configuration $\sigma'_x={\Big((w_i,y_{i1},\dots,y_{i\ell_i})\Big)}_{1\leq i\leq d-1}$ still occurs with positive probability.
\end{lemma}

We assume that there is at least one final class $C$, and first prove that necessarily $C=\{1\}$. 
Let $x\in C$ be such that $\Pb(X=x)>0$, and let $v$ be a $V$-node in $\Gamma'$ with type $x$. We assume by contradiction that $x\neq 1$. Since $0$ is singular, necessarily $x>1$. By the definition of a final class, the number $N$ of children of $v$ whose type is in $C$ is one almost surely. Hence there exists a configuration $\sigma_x={\Big((w_i,x_{i1},\dots,x_{i\ell_i})\Big)}_{1\leq i\leq d-1}$ with a positive probability to occur and in which there exists a unique couple $(i_0,j_0)$ such that $x_{i_0j_0}\in C$. We distinguish several cases:
\begin{itemize}
 \item If there exists $(k,\ell)$ such that $x_{k\ell}>x_{i_0j_0}$, we consider the new configuration $\sigma'_x$ with $y_{ij}= x_{ij}$ if $(i,j)\neq (k,\ell)$ and $y_{k\ell}= x_{i_0j_0}$. Then $N=2$ in $\sigma'_x$, and $\sigma'_x$ has positive probability to occur by Lemma \ref{lem:positive_config}, which is a contradiction.
\item If there exists $(k,\ell)$ such that $x_{k\ell}<x_{i_0j_0}$, we consider the new configuration $\sigma'_x$ with $y_{ij}= x_{ij}$ if $(i,j)\neq (i_0,j_0)$ and $y_{i_0j_0}= x_{k\ell}$. Then $N=0$ in $\sigma'_x$, and $\sigma'_x$ has positive probability to occur by Lemma \ref{lem:positive_config}, which is a contradiction.
\item Otherwise, $v$ has only one grandchild (of type $x_{i_0j_0}$) in $\Gamma'$. We first consider the case where $w_{i_0}=2$. Since $\sum_{i=1}^{d-1} (w_i-1) = x>1$, there exists $k\neq i_0$ such that $w_k\geq 2$. We construct $\sigma'_x$ by replacing $w_k$ by $w_k-1$ $E$-nodes of type $2$, each of which having a $V$-node of type $x_{i_0j_0}$ as a child. Hence $N=w_k\geq 2$ in $\sigma'_x$, and it is easy to see that $\sigma'_x$ has positive probability to occur, which is a contradiction.
\item The remaining case is when $v$ has only one grandchild (of type $x_{i_0j_0}$) in $\Gamma'$, and $w_{i_0}>2$. We construct $\sigma'_x$ by replacing $(w_{i_0},x_{i_0j_0})$ by $(w_{i_0},x_{i_0j_0},x_{i_0j_0})$. Hence $N= 2$ in $\sigma'_x$, and it is easy to see that $\sigma'_x$ has positive probability to occur, which is a contradiction.
\end{itemize}
Hence $C=\{1\}$, which means that $v$ has only one grandchild in the original branching process $\Gamma=\BP$. It implies in particular that $q_2>0$, and we are left to compute the following quantity:
\begin{eqnarray*}
 p^{(E)}(\ell=1,x=1|w=2) &=& \Indb_{\lfloor2q\rfloor\leq 0} \Pb(X=1) \\
&=& \frac{2p_2}{\lambda}\frac{2q_2}{\mu}
\end{eqnarray*}
By the definition of a final class, $p^{(E)}(\ell=1,x=1|w=2)=1$, which implies that $p_2=q_2=1$, and ends the proof.

\section{Conclusion and perspectives} \label{sec:ccl}

We studied rigorously the contagion \cite{bl95,mor} on a clustered random graph model with overlapping communities. Our random graph model allows an arbitrary distribution for the community sizes, while the heuristic study of Hackett \textit{et al.} \cite{HMGleeson:New} was done on a random graph model with communities of size three. Up to our knowledge, these are the only two studies of this epidemic model on random graphs with overlapping communities.
There are several dependencies that made this study challenging: \textit{(i)} in the epidemic model itself, since the behavior of an individual depends on the behavior of all her neighbors; \textit{(ii)} in the random graph model considered, that allows an arbitrary distribution for both the community size and the number of communities an individual belongs to.
In addition, our study provides heuristics for the contagion on the one-mode projection of a random bipartite graph with arbitrary degree distributions, which is well appropriate for modeling real-world networks \cite{NewSW01:GCbipartite,GuillaumeLatapy06}.

We showed that our epidemic is completely described by a multi-type and alternating branching process, and use a non-classical theorem on phase transitions for multi-type branching processes, referred to as Sevastyanov's theorem, to prove a phase transition for our process.
This opens the way to the study of the clustering effect on the cascade phenomenon in this case, as done in \cite{HMGleeson:New} or \cite{arxiv12}.

\thanks{The authors acknowledge the support of the French Agence Nationale de la Recherche (ANR) under reference ANR-11-JS02-005-01 (GAP project).}

\bibliographystyle{plain}

\end{document}